\documentclass[smallcondensed]{svjour3}

\usepackage{epsfig}
\usepackage{graphicx}
\usepackage{makeidx}
\usepackage{multicol}
\usepackage{verbatim}
\usepackage{subfigure}
\usepackage{mathrsfs}

\usepackage{amsmath}

\usepackage{crop}
\crop

\usepackage{amssymb}
\usepackage{graphicx}

\def\E{\mathbb{E}}

\def\ba{\begin{array}}
\def\ea{\end{array}}
\def\bi{\begin{itemize}}
\def\ei{\end{itemize}}

\begin{document}

\title{Identification of Finite Dimensional L\'evy Systems in Financial Mathematics}


\author{L.~Gerencs\'er\
\and M.~M\'anfay
}

\institute{L.~Gerencs\'er \at MTA SZTAKI, 13-17 Kende Street, Budapest, Hungary \\ Tel.: +36-1-2796138 \\ \email{gerencs@mta.sztaki.hu}
\and M.~M\'anfay \at MTA SZTAKI, 13-17 Kende Street, Budapest, Hungary; Central European University, 9 N\'ador Street, Budapest, Hungary\\ Tel.: +36-1-2796190 \\ \email{manfay@mta.sztaki.hu}
}

\date{Received: date / Accepted: date}

\maketitle




\begin{abstract}
L\'evy processes are widely used in financial mathematics to model
return data.
Price processes are then defined as a corresponding geometric
L\'evy process, implying the fact that returns are
independent. In this paper we propose an alternative class of
models allowing to describe dependence between return data. Technically such an
alternative model class is obtained by considering finite
dimensional linear stochastic SISO systems driven by a L\'evy
process. In this paper we consider a discrete-time version of this
model, focusing on the
problem of identifying the dynamics and the noise characteristics
of such a so-called L\'evy system. The special feature of this
problem is that the characteristic function (c.f.) of the driving
noise is explicitly known, possibly up to a few unknown
parameters.
We develop and analyze a variety of novel identification methods
by adapting the so-called empirical characteristic function method
(ECF) originally devised for estimating parameters of c.f.-s from
i.i.d. samples. Precise characterization of the errors of these
estimators will be given, and their asymptotic covariance matrices
will be obtained. Their potential to outperform the prediction
error method in estimating the system parameters will also be
demonstrated.



\end{abstract}



\newpage

\section{Introduction}
\label{intro}
The classical model for modelling market dynamics, namely
geometric Brownian motion, was proposed by Louis  Bacehelier
\cite{Bacehelier}.
This model is still the accepted core model despite the fact that
empirical studies revealed that its assumptions are not realistic.
For example, since price movements are induced by transactions
which can be unevenly distributed in real time, it would be more
natural to use a time changed Brownian motion to model price
dynamics. If the time change is defined by a gamma process, we
obtain the so-called VG (shorthand for Variance Gamma) process. VG
processes reproduce a number of stylized facts of real price
processes, such as fat tails and large kurtosis. It can be shown
that the above time changed Brownian process itself is a L\'evy
process. Extending the above construction novel price dynamics
have been proposed by a variety of authors, called the geometric
L\'evy processes obtained by exponentiating a L\'evy process.

A L\'evy process $(Z_t)$ is much like a Wiener process: a process
with stationary an independent increments, but discontinuities or
jumps are allowed. A good survey paper on L\'evy processes used in
financial modelling is the paper by Miyahara and Novikov,
\cite{novikov}. \cite{raible} studies several problems arising in the field of exponential L\'evy processes. For an excellent introduction to the theory of
L\'evy processes see \cite{jacod}. A key building block in the
theory of L\'evy processes is the compound Poisson process. A more
general class L\'evy process is formally obtained via
\begin{equation}
\label{eq:Levy_def}
Z_t = \int_0^t \int_{{\mathbf R}^1} x N(ds, dx),
\end{equation}
where $N(dt,dx)$ is a time-homogeneous, space-time Poisson point
process, counting the number of jumps of size $x$ at time $t$. In
this case $Z_t$ is a pure jump process, which
paradoxically
means that the L\'evy-Ito decomposition of $Z_t$ does not have a
Brownian motion component (but it may have a drift term).
The intensity of $N(dt,dx)$ is defined by $\E [N(dt,dx)],$ which is due to time homogeneity can be written as
$$\E [N(dt,dx)] = dt \cdot \nu(dx),$$
where $\nu(dx)$ is the L\'evy-measure.
The above representation given in (\ref{eq:Levy_def})
is mathematically rigorous if
\begin{equation}
\label{eq:INTEGRABILIY}\int_{{\mathbf R}^1}  \min (|x|,1) \nu(dx)
< \infty. \end{equation} Under this condition the sample paths of
$Z_t$ are of {\it finite variation}, a property supported by
empirical evidence for most indices as emphasized in \cite{CGMY}.
The characteristic function of a L\'evy process can be written in
the form
\begin{equation}
\E \left[e^{iuZ_t}\right] = e^{t \psi(u)},
\end{equation}
where $\psi(u)$ is the characteristic exponent.

The
standard model of a price process within this framework is then
\begin{equation}
\label{eq:expmodel} S_t=S_0 \exp Z_t,
\end{equation}
and $(S_t)$ is called a geometric L\'evy process. A variety of
choices for $(Z_t)$ has been proposed in the literature: it can be
a stable process, a variance Gamma (VG) process, a tempered stable
process, a special case of which is the (CGMY) process, a
hypergeometric
process or a Normal-inverse Gaussian (NIG) process. 


The motivation behind these models is the assumption that the
returns of the stock process, say $(S_{t+h}-S_t)/S_t$ are
independent and stationary. While this is an attractive
assumption, its consequences are less attractive. In particular it
follows that the variance of the price process tends to infinity,
which is certainly unnatural for, say, prices of agricultural
products.
A closer look at data in fact reveals that there is a weak
correlation between daily returns $(S_{t+1}-S_t)/S_t.$ For
example, considering data on IBM
 Coca Cola
stock prices in a period of 20 years from Nov 1990 to Nov 2010 we
found for the correlation coefficients of daily log-returns $X_t$
that
$${\rm corr}(X_t,X_{t-1})=-0.135.$$
This small, but non-negligible,
negative correlation calls for a refinement of the exponential
L\'evy model, allowing memory in the daily return process. An
intuitive empirical argument can also be given in favor of the
need for memory: namely an overreaction of the market is generally
followed by a correction, resulting in a correlation between daily
returns. The recently much studied popular Geometric fractional Brownian motion model gives
return process with non-independent increments, for more details on fractional Brownian motion see papers of T.E. Duncan, for example \cite{duncan}.

We propose to introduce a new class of models, using the
methodology of linear system theory, to capture the presence of
decaying memory. The infinitesimal increments of the logarithm of
the price process will be defined as a process $dY_t$ which is the
output of a finite dimensional stable linear SISO (shorthand for
single-input-single-output) system, driven by a L\'evy process:
\begin{equation}
\nonumber
dY_t=A dZ_t,
\end{equation}
where $A$ represents the linear mapping from input to output, and
$Z$ is a L\'evy process. For the sake of convenience we let
$-\infty < t < + \infty$. In the case of a finite dimensional
stable linear SISO system the mapping $A$ can be described by a
set of state-space equations, a well known example of such systems is defined by:
\begin{align}
dX_t&=HX_t dt + dZ_t \\
dY_t&=LX_t dt + dZ_t
\end{align}
From the above equations we get
\begin{equation}
dY_t=L\left(\int_{-\infty}^{t}e^{H\left(t-s\right)} dZ_s\right) dt
+ dZ_t.
\end{equation}

The inverse system is formally obtained as
\begin{align}
d{X}_t&=\left(H-KL\right){X}_t dt + KdY_t \\
d Z_t&=dY_t-Ld {X}_t.
\end{align}
It is assumed that both systems $A$ and $A^{-1}$ are exponentially
stable, equivalently, we assume that both $H$ and $(H-KL)$ are
stable matrices. Such a system will be called a L\'evy system.

The inverse filter has the following form:
\begin{align}
d\hat{X}_t&=\left(H-KL\right)\hat{X}_t dt + KdY_t \\
d\varepsilon_t&=dY_t-Ld\hat{X}_t.
\end{align}

Having defined the infinitesimal increments of the logarithm of
the price process we define the price process according to
(\ref{eq:expmodel}):
$$S_t=S_0 \exp Y_t.$$

In the statistical analysis of such systems, both the system
dynamics and the fine characteristics of $(Z_t)$ are to be
identified. The first difficulty of applying a maximum-likelihood
(ML) method lies in the fact that there is no natural reference
measure in the space of sample paths. In addition, the
computation of the Radon-Nikodym derivative is practically not
feasible since $\int_{-\infty}^{t} e^{H(t-s)} dZ_s $ is not even a L\'evy process.



To avoid this problem we consider an alternative discrete-time
model class, where the daily log-returns $\Delta y_n$ are defined
via a discrete time finite dimensional system
\begin{equation}
\label{eq:disc_levy} \Delta y_n=A~\Delta Z_n,
\end{equation}
where $A$ represents the linear mapping from input to output, and
$\Delta Z_n$ is the increment of a L\'evy process $Z$ over an
interval $[(n-1)h,nh)$, with some fixed $h >0$. For the sake of
convenience we let $-\infty < n < + \infty$.
A state space equation for this model is given by
\begin{align}
\Delta X_{n+1}&=H \Delta X_n + \Delta Z_n \\
\Delta Y_n&=L \Delta X_n + \Delta Z_n.
\end{align}
We will call this model a discrete time finite dimensional L\'evy
system. Assume that $A=A(\theta^*)$ where $\theta^*$ is an unknown
parameter-vector, and similarly, let $\nu(dx)=\nu(dx,\eta^*),$
where $\eta^*$ denotes an unknown parameter-vector. The ranges of
of $\theta^*$ and $\eta^*$ are assumed to be known. The
fundamental problem to be discussed in this paper is to identify
this system and to establish sharp results for the error of the
estimator.

If we knew the probability density function of the noise $\Delta
Z_n$ then we could apply an ML (Maximum Likelihood) estimation
method, and establish sharp results for the estimation error, see
\cite{LSS_ML_GL+MGY+RZ}. The challenge of the present problem is
that it is the characteristic function of the noise that is
explicitly given. A natural approach to solve this problem is to
combine techniques of system identification with the empirical
characteristic function (ECF) method widely used in finance to
analyze i.i.d. data. Before going into further details we present
a few examples of L\'evy processes used in finance.


\section{L\'evy processes in finance}
\label{sec:levy_fin}

To model the increments of the logarithm of a price process a wide
range of geometric L\'evy processes has been proposed by a variety
of authors.
%
Mandelbrot suggested to use $\alpha$-stable process to model the
price dynamics of wool, see \cite{MANDELBROT}. An $\alpha$-stable
with $0 < \alpha < 2$ is defined via the L\'evy measure
\begin{equation}
\nonumber
\nu(dx) = C^- |x|^{-1-\alpha} {\mathbf 1}_{x < 0 } dx+ C^+ |x|^{-1-\alpha} {\mathbf 1}_{x
> 0 } dx.
\end{equation}
%
A recently widely studied class of L\'evy processes is the CGMY
process due to Carr, Geman, Madan and Yor \cite{CARR-EMPINV}. It
is obtained by setting $C^- = C^+$, and then, separately for $x
>0$ and $x<0$, multiplying the
L\'evy-density of the original symmetric stable process with a
decreasing exponential. The corresponding L\'evy-measure, using
standard parametrization, is of the form:
\begin{equation}
\nonumber
\nu(dx) = {\frac {C e^{-G |x|}} {|x|^{1+Y}}}  {\mathbf 1}_{x < 0 } dx +
{\frac {C e^{-M x}} {|x|^{1+Y}}} {\mathbf 1}_{x
> 0 } dx,
\end{equation}
where $C,G,M >0$, and $0 < Y <2$. Intuitively, $C$ controls the
level of activity, $G$ and $M$ together control skewness.
Typically $G > M$ reflecting the fact that prices tend to increase
rather than decrease. $Y$ controls the density of small jumps,
i.e. the fine structure. For $Y<1$ the integrability condition (\ref{eq:INTEGRABILIY}) is satisfied, thus corresponding L\'evy
process is of finite variation. The characteristic exponent of the
CGMY process is given by
\begin{equation}
\psi(u) = C \Gamma (-Y) \left((M-iu)^Y- M^Y + (G+iu)^Y- G^Y\right),
\end{equation}
where $\Gamma$ denotes the gamma-function.

Allowing $C$ and $Y$ to take on different values for $x
>0$ and $x<0$ we get a more general class of processes called tempered stable
process. see cite{}.

Formally setting $Y=0$ we get the L\'evy density of the so-called
Variance Gamma process (VG for short) that has been proposed by
Madan, Carr and Chang \cite{CARR-VG}. The VG process is a time
changed Brownian motion when the time change is a gamma process,
which itself is a L\'evy process, obtained by properly extending
the definition of the inverse of a Poisson process from natural
numbers to positive reals. Thus we can write
$$VG(t)=W_{\theta,\sigma}(\gamma_{\mu,\nu}(t)),$$
where $W_{\theta,\sigma}(t)=\theta t+\sigma W(t),$ with $W$ being the standard Wiener process,
and $\gamma$ is a gamma process with mean rate $\mu,$ and variance rate $\nu$, see \cite{CARR-VG}.

Its characteristic function is given by
$$
\varphi_{VG(t)}(u)=\left(1-iu\theta \nu+u^2 \sigma^2 \nu/2 \right)^{-t/\nu}.
$$
This can be obtained by a formal limiting procedure taking into
account the characteristic exponent given by (2.1) and taking $Y
\rightarrow 0.$

 The knowledge of the explicit form of the characteristic function is a common feature of distributions in finance. This is the case for tempered stable and related processes, see \cite{terdik}. We will focus on the CGMY process.
%

\section{Discrete time L\'evy systems}
\label{sec:disc_levy}
A discrete time finite dimensional L\'evy system is defined as
\begin{equation}
\label{eq:disc_levy2}
\Delta y_n=A(\theta^*) \Delta Z_n,
\end{equation}
where $\Delta Z_n$ is the increment of a L\'evy process $Z$ over
an interval $[(n-1)h,nh)$ with $\E \left[\Delta Z_n\right]=0,$ a property to be removed later and $h>0$ is a fix sampling interval. The L\'evy-measure of $Z$ will be
denoted by
  $\nu(dx)=\nu(dx,\eta^*),$ where $\eta^*$ denotes an unknown
  parameter-vector, for example for a CGMY process $\eta^*=(C,G,M,Y).$ The range of $\eta^*$ is assumed to be known.

\textbf{Condition 1}  We assume that
  \begin{equation}
  \label{eq:moments}
  \int_{|x| \ge 1}|x|^q\nu(dx)<+\infty
  \end{equation}
for all $1 \leq q \leq Q$ with some constant $Q$.

Note that Condition 1 holds with $Q=\infty$ in our benchmark examples.
%
%
%
%
%
Let $D_{\rho}$ and $D_{\rho}^*$ be compact domains such that $\rho^* \in D_{\rho}^* \subset \text{int } D_{\rho}$ and $D_{\rho} \subset G_{\rho}.$

\textbf{Condition 2}
 $A(\theta)$ is assumed to be exponentially stable and exponentially inverse stable
for $\theta\in G_{\theta}\subset \mathbb{R}^p,$ where $G_{\theta}$ is a known open
set.

A system is exponentially stable if all the eigenvalues of $A$ have strictly
negative real parts. The application of the ML
method would solve the full identification problem along standard
lines, assuming that the density function of $\Delta Z_n $ is
known, see \cite{gerencs_ML}, which is unfortunately not the case. The objective of this paper is to
present a combination of advanced techniques in systems
identification with a specific statistical technique, widely used
in the context in finance, called the ECF (shorthand for empirical characteristic function) method. The ECF method was originally designed for i.i.d. samples and A.~Feuerverger and P.~McDunnogh \cite{FEUER+MC} showed that it can be interpreted as the Fourier transform of an ML method.

\section{Three identification problems}
\label{sec:problems}
In this section we formulate three identification problems related
to discrete-time, finite dimensional L\'evy systems, and sketch a
possible path to their solution. The first, simplest problem is
seemingly of
mere technical interest:

\textit{Known system parameters, unknown noise parameters.} In
this case define and compute
$$\varepsilon_n (\theta^*)=A^{-1}(\theta^*)\Delta y_n=A^{-1}(\theta^*)A(\theta^*)\Delta Z_n=\Delta
Z_n,$$ assuming, for the sake of simplicity, that $\Delta Z_n =
\varepsilon_n (\theta^*) = 0$ for $n \le 0.$ After that we can
apply the ECF method for i.i.d. samples to obtain the estimation
of $\eta^*.$ This simple solution will be the base of the identification method presented in Section \ref{sec:mixed}.


\textit{Known noise parameters, unknown system parameters.} This
is the simplest, technically interesting and non-trivial problem.
If we knew the probability density function of the noise, say $f$,
we could obtain the maximum likelihood estimate of $\theta^*$ via
solving
\begin{equation}
\sum_{n=1}^{N} f_{\theta}\left(\varepsilon_n(\theta),\eta^*\right)=0,
\end{equation}
where
\begin{equation}
\varepsilon_n(\theta)=A^{-1}(\theta)\Delta y_n
\end{equation}
is the estimated innovation process of a SISO system, see \cite{gerencs_ML}.

Under certain conditions the asymptotic covariance matrix of the ML estimate $\hat{\theta}_N$ is
\begin{equation*}
\Sigma_{ML}=\mu \left(R^*\right)^{-1},
\end{equation*}
where
$$
\mu=
\E \left[\left(\frac{f'(\Delta Z_n,\eta^*)}{f(\Delta Z_n,\eta^*)}\right)^2\right],$$
with $f'$ being the derivative of $f$ w.r.t the first variable and
$$
R^*=\lim_{n \rightarrow \infty} \E\left[\varepsilon_{n\theta}(\theta^*)\varepsilon^T_{n\theta}(\theta^*)\right].
$$
In our case, the p.d.f. of the noise distribution is not known.
One might apply the prediction error method to estimate the system
dynamics, i.e. $\theta^*.$ However, we will show, in the case of
CGMY noise, that we may estimate $\theta^*$ in a more efficient
way using an appropriate adaptation of the ECF method. In fact,
this result is a special case of a more general result obtained
for the general problem to be described in the next subsection.

\textit{Both the system parameters and the noise parameters are
unknown.} The first method that we propose is quite
straightforward: we estimate the system parameters using a PE
method, then, using a certainty equivalence argument, we estimate
the innovation process by inverting the system using the estimated
parameters. Then, we estimate the noise parameters using ECF
method for i.i.d. sequences. This method will be studied in
Section \ref{sec:mixed}.

The second method, which is the main subject of this paper,
estimates both the system parameters and noise parameters using an
ECF method. First, an parameter-dependent, estimated innovation
process $\varepsilon_n(\theta)$ is defined, then the
characteristic function of the noise is fitted to empirical data
defined in terms of $\varepsilon_n(\theta)$. Thus we get a score
function that depends on both $\theta$ and $\eta.$

The third method applies an extension of the ECF method using the blocks of the time-series of unprocessed data $\left\{\Delta
y_n\right\}_{n=0}^{\infty}.$ More details can be found in the Discussion.

\section{Single term ECF method}
\label{sec:ECF}
The ECF method has been widely used in finance as an alternative to the ML Method, assuming i.i.d. returns \cite{CARRASCO-CGMM}, \cite{CARRASCO-EFFEMPIRCHAR}, \cite{YU-EMPIRCHARFUN}. We adapt this technique to the problem of identifying the discrete-time L\'evy system described in (\ref{eq:disc_levy}).
%
%
%
%
Fix a realization of $A$ in its innovation form, i.e. assume that $A$ and its inverse are exponentially stable. The estimated innovation process ($\varepsilon_n(\theta)$)
is defined via the inverse filter:
\begin{align}
d\hat{X}_t&=\left(H-KL\right)\hat{X}_t dt + KdY_t \\
d\varepsilon_t&=dY_t-Ld\hat{X}_t,
\end{align}
for continuous time models. For discrete time L\'evy systems we define the innovation process by
\begin{equation}
\label{eq:inverse}
\varepsilon_n (\theta)=A^{-1}(\theta) \Delta y_n,
\end{equation}
with zero initial conditions and $\theta \in D_{\rho}.$ Let $\varepsilon^*_n(\theta)$ denote the stationary solution of (\ref{eq:inverse}) when $-\infty<n<\infty.$ In general, the notation $(.)^*$ will be used throughout this paper if the
corresponding stochastic process is obtained by passing through a
stationary process through an exponentially stable linear filter
starting at $- \infty$, as opposed to initializing the filter at
time $0$ with some arbitrary initial condition, which is typically
zero. Then we have for $n \ge 0$
\begin{equation}
\label{eq:statdiff}
\varepsilon^*_n(\theta)=\varepsilon_n(\theta)+r_n,
\end{equation}
where $r_n=O_M^Q(\alpha^n)$ with some $0<\alpha<1$, meaning that for all $1 \leq q <Q$
$$ \sup_{n} \alpha^{-n} \E^{1/q} \left|r_n\right|^q < \infty.$$
We will use this notation in a more general way:
\begin{definition}
For a stochastic process $X_n,$ and a function
$f:\mathbb{Z}\rightarrow \mathbb{R}^+$ we say that
$$ X_n=O_{M}^{Q}(f(n)) $$ if for all $1 \leq q \leq Q$
$$ \sup_{n} \frac{\E^{1/q}\left|X_n\right|^q}{f(n)} < \infty $$
holds.
\end{definition}
%
%
The score functions to be used following the basic idea of the ECF method are defined as
\begin{align}
    h_n(u;\theta,\eta)= e^{iu \varepsilon_n (\theta)}-\varphi(u,\eta) \\
    h^*_n(u;\theta,\eta)= e^{iu \varepsilon^*_n (\theta)}-\varphi(u,\eta)
\end{align}
with $u \in \mathbb{R}.$
These are indeed appropriate score functions, since we obviously have
$$ \E   \left[h^*_n(u;\theta^*,\eta^*)\right]=0,$$
and $$ h_n(u;\theta^*,\eta^*)=h^*_n(u;\theta^*,\eta^*)+O^Q_M(\alpha_n).$$
While $h_n$ is the function that can be computed in practice, $h^*_n$ is easier to handle, because its stationarity.
Following the philosophy of the ECF method take a fix set $u_i$-s, and define the
$k$-dimensional vector
$$h_n(\theta,\eta)=\left(h_n(u_1;\theta,\eta),\dots, h_n(u_k;\theta,\eta)\right)^T.$$
Let $K>0$ be a fixed symmetric, positve definite $k \times k$ weighting matrix. Since the system of equations
$$
h_n(\theta,\eta)=0 \quad n=1,\ldots,N
$$
is overdetermined we seek a least-square solution. Therefore we define
the cost functions as

$$V_N=V_N(\theta,\eta)=\sum_{n=1}^N |K^{-1/2}h_n(\theta,\eta)|^2,$$

$$V^*_N=V^*_N(\theta,\eta)=\sum_{n=1}^N |K^{-1/2}h^*_n(\theta,\eta)|^2,$$
and by solving
\begin{eqnarray}
V_{N \theta}(\theta,\eta)=0 \\
V_{N \eta}(\theta,\eta)=0
\label{eq:score2_2}
\end{eqnarray}
we obtain the estimation $\hat{\theta}_N$ and $\hat{\eta}_N$ of $\theta^*$ and $\eta^*$, respectively.

\section{Analysis}
\label{sec:ECF_proofs}
Differentiating $V^*_N$ w.r.t $\theta$ and $\eta$ we get the equations
%
%
\begin{align}\label{eq:Q_der_1}
V^*_{N \theta}(\theta,\eta)=\sum_{n=1}^N \left(h^{T*}_{n \theta}(\theta,\eta)K^{-1}\bar{h^*}_n(\theta,\eta)+ h^{T*}_{n}(\theta,\eta)K^{-1}\bar{h^*}_{n \theta}(\theta,\eta)\right) &= 0,\\ \label{eq:Q_der_2}
V^*_{N \eta}(\theta,\eta)=\sum_{n=1}^N \left(h^{T*}_{n \eta}(\theta,\eta)K^{-1}\bar{h^*}_n(\theta,\eta)+ h_{n}^{T*}(\theta,\eta)K^{-1}\bar{h^*}_{n \eta}(\theta,\eta) \right)&= 0,
\end{align}
where $\bar{h}$ is the conjugate of $h.$
Note that, setting $\theta=\theta^*$, the second equation is just
the optimality condition of the ECF method for i.i.d. samples \cite{CARRASCO-EFFEMPIRCHAR}. As for the first equation, the derivative of the score function $h$ with respect to $\theta$ is
\begin{equation}
    h_{n \theta}(u,\theta,\eta)=e^{iu\varepsilon_n(\theta)} iu \varepsilon_{n \theta}(\theta).
\end{equation}
Hence in the first equation $h_{n\theta}(\theta,\eta)$ and
$h_n(\theta,\eta)$ are not independent. However, the next lemma shows that their stationary approximation, $h^*_{n \theta}(\theta,\eta)$ and $h^*_n(\theta,\eta),$ are uncorrelated.

\begin{lemma}
\label{lemma:exp} For any $\eta$ we have $\E \left[V^*_{N \theta}(\theta^*,\eta)\right]=0$, and in addition $\E \left[V^*_{N \eta}(\theta^*,\eta^*)\right]=0.$
\end{lemma}

\begin{proof}



Consider the $n^{th}$ term in (\ref{eq:Q_der_1}). We have
\begin{align}\nonumber
& \E\left(h^{*T}_{n\theta}(\theta^*,\eta)K^{-1}\bar{h}^*_n(\theta^*,\eta)\right)= \\
&=\sum_{l,m=1}^k K_{l,m}^{-1} \mathbb{E} \left[\left(e^{iu_l\varepsilon^*_n(\theta^*)} iu_l \varepsilon^*_{n\theta}(\theta^*)\right) \left(e^{-iu_m \varepsilon^*_{n} (\theta^*)}-\varphi(-u_m,\eta)\right)\right].
\end{align}
Compute the $(l,m)^{th}$ term using the tower law:
\begin{align}
\mathbb{E} \left[\left(e^{iu_l\varepsilon^*_n(\theta^*)} iu_l \varepsilon^*_{n \theta}(\theta^*)\right) \left(e^{-iu_m \varepsilon^*_n (\theta^*)}-\varphi(-u_m,\eta)\right)\right]= \nonumber \\
=\mathbb{E} \left[ \mathbb{E} \left[ \left(e^{iu_l\varepsilon^*_n(\theta^*)} iu_l \varepsilon^*_{n \theta}(\theta^*)\right) \left(e^{-iu_m \varepsilon^*_n (\theta^*)}-\varphi(-u_m,\eta)\right)| \mathscr{F}_{n-1}^{\Delta Z}\right]\right], \label{eq:E_calc2}
\end{align}
 where $\mathscr{F}_{n-1}^{\Delta Z}=\sigma \left\{ \Delta Z_k: k \leq n-1 \right\} .$
Here we used that $\overline{\varphi}(u,\eta)=\varphi(-u,\eta).$ Due to the fact that $\varepsilon^*_{n \theta}(\theta^*)$ is $\mathscr{F}_{n-1}^{\Delta Z}$ measurable, (\ref{eq:E_calc2}) can be written as
\begin{equation} \nonumber
\begin{split}
&\mathbb{E} \left[ iu_l \varepsilon^*_{n\theta}(\theta^*)\E\left[e^{i(u_l-u_m)\varepsilon^*_n(\theta^*)}-e^{iu_l\varepsilon^*_n(\theta^*)}\varphi(-u_m,\eta) | \mathscr{F}_{n-1}^{\Delta Z}\right]\right]= \\
&\mathbb{E} \left[ iu_l \varepsilon^*_{n \theta}(\theta^*)\left(\varphi(u_l-u_m,\eta^*)-\varphi(u_l,\eta^*)\varphi(-u_m,\eta)\right)\right]= \\
&\left(\varphi(u_l-u_m,\eta^*)-\varphi(u_l,\eta^*)\varphi(-u_m,\eta)\right) \mathbb{E} \left[ iu_l \varepsilon^*_{n \theta}(\theta^*)\right]=0.
\end{split}
\end{equation}
To reduce the last equation we used that $\E \left[\Delta Z_n\right]=0.$

Similarly for the $n^{th}$ term of (\ref{eq:Q_der_2}) we have
$$h^*_{n \eta}(u,\theta,\eta)=-\varphi_{\eta}(u,\eta),$$  which is non-random implying that
\begin{equation}\nonumber
\E \left[h^*_{n\eta}(u,\theta^*,\eta^*)K^{-1}\bar{h}^*_n(u,\theta^*,\eta^*)\right]=0.
\end{equation}
\qed
\end{proof}

The previous lemma also shows that the gradient of $V_N(\theta,\eta)$ serves as an alternative score function.
The following corollary is implied by the fact that
\begin{equation}
\E\left[h_{n\theta}^T(\theta^*,\eta)K^{-1}\bar{h}_n(\theta^*,\eta)\right]=\E\left[h_{n\theta}^{*T}(\theta^*,\eta)K^{-1}\bar{h}^*_n(\theta^*,\eta)\right]+O_M^Q(\alpha^n).
\end{equation}

\begin{corollary}
For any $\eta$ we have $\E \left[V_{N \theta}(\theta^*,\eta)\right]=O^{Q}_M(\alpha^N)$, and in addition \\ $\E \left[V_{N \eta}(\theta^*,\eta^*)\right]=O^{Q}_M(\alpha^N).$
\end{corollary}

Define $\rho=(\theta,\eta),$ and define the asymptotic cost function by
$$ W(\theta, \eta)=W(\rho)=\E \left|K^{-1/2} h^*_n(\rho)\right|^2.$$
\textbf{Condition 3}
The equation $W_{\rho}(\rho)=0$ has a unique solution in $D_{\rho}^*.$

A crucial object is the Hessian of $W$ at $\rho=\rho^*:$
$$ R^*=W_{\rho \rho}(\rho^*). $$
It is easy to see that
$$R^*=
\begin{pmatrix}
W_{\theta \theta}(\theta^*) & 0 \\
0 & W_{\eta \eta}(\eta^*)
\end{pmatrix}
$$ is block diagonal matrix.









The following result provides a precise characterization of the estimation error:
\begin{theorem}
\label{th:difference}
Under Conditions 1,2 and 3 we have
$$ \hat{\rho}_N-\rho^*=-(R^*)^{-1}\frac{1}{N} V_{N \rho }(\rho^*)+O^{Q/(2(p+q))}_M(N^{-1})$$
\end{theorem}
First, we prove some lemmas that will be used in the proof of Theorem \ref{th:difference}. For the definition of $L$-mixing processes and for other corresponding definitions and theorems see the Appendix.

\begin{lemma}
Under Conditions 1,2,3 processes
$\varepsilon_n(\theta),\varepsilon_{n\theta}(\theta)$ and $\varepsilon_{n\theta\theta}(\theta)$ are $L$-mixing uniformly of order $Q$.
\end{lemma}

\begin{proof}
First, note that since $\Delta y_n=\sum_{i=0}^q a_i(\theta^*)\Delta Z_i,$ holds, $\Delta y_n$ is a linear combination of $L$-mixing processes of order $Q.$ Using the fact that an uniformly exponentially stable filter with $L$-mixing input produces an uniformly $L$-mixing output \cite{gerencs_mixing} we get that $\Delta y_n$ is $L$-mixing processes of order $Q$ for each $n.$  The innovation process and its derivatives with respect to $\theta$ can be written as
\begin{align*}
\varepsilon_n(\theta)&=A^{-1}(\theta) \Delta y_i \\
\varepsilon_{n\theta}(\theta)&=A_{\theta}^{-1}(\theta) \Delta y_i \\
\varepsilon_{n\theta\theta}(\theta)&=A_{\theta \theta}^{-1}(\theta) \Delta y_i.
\end{align*}
Again, since $A^{-1}(\theta)$ and its derivative with respect to $\theta$ are uniformly exponentially stable we conclude the lemma.
\qed
\end{proof}

\begin{lemma}
Suppose that 
Conditions 1,2,3 hold.
Then for any given $d>0$ the equation $V_{N \rho}(\rho)=0$ has a unique solution in $D_{\rho}$ and it is in the sphere $S=\left\{\rho:\left|\rho-\rho^*\right|<d\right\}$ with probability at least $1-O(N^{-s})$ for any $0<s\leq Q/2.$ Furthermore the constant $C$ in $O(N^{-s})=CN^{-s}$ depends only on $d$ and $s.$
\end{lemma}

\begin{proof}
First, note that since $\varepsilon_n,\varepsilon_{n\theta}$ and $\varepsilon_{n\theta\theta}$ are $L$-mixing processes uniformly of order $Q$, the processes $h_n,h_{n\rho}$ and $h_{n\rho\rho}$ are $L$-mixing uniformly of order $Q/2$ as well. It follows that the process
\begin{equation}
u_n(\rho)=\frac{\partial}{\partial \rho}\left(h^T_n(\rho)K^{-1/2}\bar{h}_n(\rho)\right)(\rho)-W_{\rho}(\rho)
\end{equation}
and its derivative with respect to $\rho$ are $L$-mixing uniformly of order $Q/2$.

$\E\left[u^*_n(\rho)\right]=0$ implies $\E\left[u_n(\rho)\right]=O^{Q/2}_M(\alpha^n)$ uniformly in $\rho$ and hence following Theorem \ref{thm:sup} we have for
\begin{align}
\delta V_{N\rho}= \sup_{\rho \in D_{\rho}} \left|\frac{1}{N} V_{N\rho}(\rho)-W_{\rho}(\rho)\right|, \\
\delta V_{N\rho\rho}= \sup_{\rho \in D_{\rho}} \left|\frac{1}{N} V_{N\rho\rho}(\rho)-W_{\rho\rho}(\rho)\right|,
\end{align}
$\delta V_{N\rho}=O^{Q/(2(p+q))}_M(N^{-1/2})$ and $\delta V_{N\rho\rho}=O^{Q/(2(p+q))}_M(N^{-1/2}).$
Thus, $$P(\delta V_{N\rho}>d)=O(N^{-s})$$
with any $0<s \leq Q/(4p+4q)$ by Markov's inequality. Applying the same argument yields
$$P(\delta V_{N\rho\rho}>d'')=O(N^{-s}),$$ for any $d''>0,$ and any $0<s \leq Q/(4p+4q).$

Suppose now that equation $V_{N \rho}(\rho)=0$ has a solution outside $S.$ Define
$$ d'=\inf \left\{\left|W_{\rho}(\rho)\right|:\rho \in D_{\rho}, \left|\rho-\rho^*\right|\geq d\right\}>0, $$
since $W_{\rho}$ is continuous and $D_{\rho}$ is compact. It follows that $\delta V_{N \rho}>d'$, and we have seen that this event has probability $O(N^{-s}).$
 So for $\Omega_N=\left\{\delta V_{N\rho}>d, \delta V_{N\rho\rho}>d'\right\}$ we have $$P(\Omega_N)>1-O(N^{-s})$$ with any $0<s \leq Q/2.$
The equation $W_{\rho}(\rho)$ has a unique solution in $D_{\rho}.$ Hence by using the implicit function theorem, see Theorem \ref{thm:implicit}, one can easily conclude that $V_{N \rho}(\rho)=0$ has a unique solution if $d'$ and $d''$ are sufficiently small.

\qed

\end{proof}

\begin{lemma}\label{lemma:N-1/2}
Under Conditions 1,2,3 we have $\hat{\rho}_N-\rho^{*}=O^{Q/2}_M(N^{-1/2}).$
\end{lemma}

\begin{proof}

We have
\begin{equation}\label{eq:taylor}
0=V_{N \rho} \left(\hat{\rho}_N\right)=V_{N \rho} \left(\rho^*\right)+ \overline{V}_{N\rho \rho} \left(\hat{\rho}-\rho^*\right),
\end{equation}
where $$ \overline{V}_{N\rho \rho}= \int_{0}^{1} V_{N\rho \rho}\left(\left(1-\lambda\right)\rho^*+\lambda \hat{\rho}_N\right) d\lambda .$$
Since $$\varepsilon_n(\theta^*)=\Delta Z_k +O_M^Q(\alpha^n) $$ for some $\left|\alpha\right|<1$, and with $u_n=\frac{\partial}{\partial \rho}\left|h_n(\rho)K^{-1/2}\right|^2$ using the inequality in Theorem \ref{thm:ineq} with $f_n=1$ and $q \leq Q$ from
\begin{equation}
\E^{1/q} \left|\sum_{n=1}^{N} u_n\right|^q \leq C_q N^{1/2} M^{1/2}_{q}(u) \Gamma^{1/2}_{q}(u)
\end{equation}
we conclude $V_{N \rho}(\rho^*)=O_M^{Q/2}(N^{1/2}).$
Let
$$
\overline{W}_{N\rho \rho}=\int_{0}^{1} W_{\rho \rho}\left(\left(1-\lambda\right)\rho^*+\lambda \hat{\rho}_N\right) d\lambda.
$$
%
%
$W$ is a smooth function, hence
\begin{equation}\label{eq:smooth}
\left\|W_{\rho \rho}\left(\rho^*+\lambda\left(\hat{\rho}_N-\rho^*\right)\right)-W_{\rho \rho}(\rho^*)\right\|<c\left|\hat{\rho}_N-\rho^*\right|<cd.
\end{equation}
Clearly $W_{\rho\rho}(\rho^*)$ is positive definite, hence $\overline{W}_{N\rho\rho}>cI,$ with some $c>0.$ Since on $\Omega_N$

$$\left\|\frac{1}{N}\overline{V}_{N\rho\rho}-\overline{W}_{N\rho\rho}\right\|<d'$$
holds, choosing $d'$ sufficiently small yields

\begin{equation}
\label{eq:lambdamin}
\lambda_{\min}\left(\frac{1}{N}\overline{V}_{N\rho\rho}\right)>c
\end{equation}
on $\Omega_N,$ where $\lambda_{\min}(M)$ denotes the smallest eigenvalue of $M.$ Thus $\left\|\overline{V}^{-1}_{N\rho\rho}\right\|<cN^{-1}$ on $\Omega_N.$ Then using ($\ref{eq:taylor}$) we get that

$$\chi_{\Omega_N}\left(\hat{\rho}_N-\rho^*\right)=O^{Q/2}_M(N^{-1/2}).$$
Furthermore, since $P(\Omega^C_N)=O(N^{-s})$ for any $0<s \leq Q/2$, the lemma follows.
\qed

\end{proof}
Now we are ready to prove Theorem \ref{th:difference}.

\begin{proof}
Using the previous lemma one can improve ($\ref{eq:smooth}$):
$$
\left\|W_{\rho \rho}\left(\rho^*+\lambda\left(\hat{\rho}_N-\rho^*\right)\right)-W_{\rho \rho}(\rho^*)\right\|<c\left|\hat{\rho}_N-\rho^*\right|=O^{Q/2}_M(N^{-1/2}),
$$
and after integration with respect to $\lambda$ we get
\begin{equation}\label{eq:dif1}
\left\|\overline{W}_{N\rho\rho}-W_{\rho\rho}(\rho^*)\right\|=O^{Q/2}_M(N^{-1/2}).
\end{equation}
Since $\delta V_{N\rho\rho}=O^{Q/(2(p+q))}_M(N^{-1/2})$, it implies
\begin{equation}\label{eq:dif2}
\left\|\frac{1}{N}\overline{V}_{N \rho\rho}-\overline{W}_{\rho\rho}\right\|=O^{Q/(2(p+q))}_M(N^{-1/2}).
\end{equation}
Hence by triangle inequality from ($\ref{eq:dif1}$) and ($\ref{eq:dif2}$)
\begin{equation}
\label{eq:dif3}
\left\|\frac{1}{N}\overline{V}_{N\rho\rho}-W_{\rho\rho}(\rho^*)\right\|=O^{Q/(2(p+q))}_M(N^{-1/2})
\end{equation}
follows. From (\ref{eq:lambdamin}) and (\ref{eq:dif3}) we get
\begin{equation}
\chi_{\Omega_N}\left\|\overline{V}^{-1}_{N\rho\rho}-\frac{1}{N}W^{-1}_{\rho\rho}(\rho^*)\right\|=O^{Q/(2(p+q))}_M(N^{-3/2}).
\end{equation}

Finally,
\begin{align*}
    \chi_{\Omega_N}\left(\hat{\rho}_N-\rho^*\right)=-\chi_{\Omega_N}\overline{V}^{-1}_{N\rho\rho}V_{N\rho}(\rho^*)= \\
    -\chi_{\Omega_N}\left(\frac{1}{N}W^{-1}_{\rho\rho}(\rho^*)+O^{Q/(2(p+q))}_M(N^{-3/2})\right) V_{N\rho}(\rho^*)= \\
    -\chi_{\Omega_N}\frac{1}{N}W^{-1}_{\rho\rho}(\rho^*) V_{N\rho}(\rho^*)+O^{Q/(2(p+q))}_M(N^{-1})
\end{align*}
Since $\chi_{\Omega_N}=1-O^Q_M(N^{-s})$ for any $0<s \leq Q/(4p+4q),$ from the last expression reads as
\begin{equation*}
-\left(R^*\right)^{-1}\frac{1}{N}V_{N\rho}(\rho^*)+O^{Q/(2(p+q))}_M(N^{-1})
\end{equation*}
\qed
\end{proof}
The following theorem provides an explicit expression for the Hessian of $W:$
\begin{theorem}\label{lemma:R*} Under Conditions 1,2,3 we have
$$R^*=
\begin{pmatrix}
W_{\theta \theta}(\theta^*) & 0 \\
0 & W_{\eta \eta}(\eta^*)
\end{pmatrix},
$$
i.e. $R^*$ is block diagonal, and here
\begin{equation*}
W_{\theta \theta}(\theta^*)=w ~\E\left[\varepsilon^{*}_{n \theta}(\theta^*)\varepsilon^{T*}_{n \theta}(\theta^*)\right],
\end{equation*}
with
\begin{eqnarray*}
w=\sum_{l,m=1}^k K^{-1}_{l,m}
\big((u^2_l+u^2_m)\varphi(u_l,\eta^*)\varphi(-u_m,\eta^*)-
(u_l-u_m)^2\varphi(u_l-u_m,\eta^*)\big),
\end{eqnarray*}
and
\begin{equation*}
\begin{split}
(W_{\eta \eta})_{j,j'}(\eta^*)=\sum_{l,m=1}^k K^{-1}_{l,m} \big(\varphi_{\eta_j}(u_l,\eta^*)\varphi_{\eta_j'}(-u_m,\eta^*)+
\varphi_{\eta_j'}(u_l,\eta^*)\varphi_{\eta_j}(-u_m,\eta^*)\big).
\end{split}
\end{equation*}
\end{theorem}

\begin{proof}

First let $j,j' \leq \text{dim } \theta=p,$ then an entry $\left(R^*\right)_{j,j'}$ of $R^*$ is
\begin{equation*}
\begin{split}
&\E \left[\frac{\partial^2}{\partial \theta_{j} \partial \theta_{j'}}\sum_{l,m=1}^k K^{-1}_{l,m}
\left(e^{i u_l \varepsilon^*_n(\theta)}-\varphi(u_l,\eta)\right)\left(e^{-i u_m \varepsilon^*_n(\theta)}-\varphi(-u_m,\eta)\right)\right|_{\substack{\theta=\theta^* \\ \eta=\eta^*}} \Bigg]
\end{split}
\end{equation*}
Carrying out differentiation yields
\begin{eqnarray*}
&\E \Bigg[ \sum_{l,m=1}^k K^{-1}_{l,m} \left(e^{i u_l \varepsilon^{*}_n(\theta^*)}(iu_l)^2 \varepsilon^{*}_{n \theta_{j}}(\theta^*)~\varepsilon^{*}_{n \theta_{j'}}(\theta^*)+e^{i u_l \varepsilon^{*}_n(\theta^*)}(iu_l)~ \varepsilon^{*}_{n \theta_{j} \theta_{j'}}(\theta^*)\right) \\ &\times \left(e^{-i u_m \varepsilon^*_n(\theta^*)}-\varphi(-u_m,\eta^*)\right)\Bigg]+
\\&+ \E \Bigg[ \sum_{l,m=1}^k K^{-1}_{l,m} \left(e^{i u_l \varepsilon^{*}_n(\theta^*)}(iu_l)~ \varepsilon^{*}_{n \theta_{j}}(\theta^*)\right)\left(e^{-i u_m \varepsilon^{*}_n(\theta^*)}(-iu_m)~ \varepsilon^{*}_{n \theta_{j'}}(\theta^*)\right)+
\\ &+\sum_{l,m=1}^k K^{-1}_{l,m} \left(e^{i u_l \varepsilon^{*}_n(\theta^*)}(iu_l)~ \varepsilon^{*}_{n \theta_{j'}}(\theta^*)\right)\left(e^{-i u_m \varepsilon^{*}_n(\theta^*)}(-iu_m)~ \varepsilon^{*}_{n \theta_{j}}(\theta^*)\right)+ \Bigg]
\\ &+ \E \Bigg[\sum_{l,m=1}^k K^{-1}_{l,m} \left(e^{i u_l \varepsilon^*_n(\theta^*)}-\varphi(u_l,\eta^*)\right) \times \\
&\left(e^{-i u_m \varepsilon^{*}_n(\theta^*)}(-iu_m)^2 \varepsilon^{*}_{n \theta_{j}}(\theta^*)~\varepsilon^{*}_{n \theta_{j'}}(\theta^*)+e^{-i u_m \varepsilon^{*}_n(\theta^*)}(-iu_m)~ \varepsilon^{*}_{n \theta_{j} \theta_{j'}}(\theta^*)\right) \Bigg]
\end{eqnarray*}
Now we use, like in the proof of Lemma \ref{lemma:exp}, the tower rule and that $\varepsilon^*_{n \theta}(\theta^*)$ is $\mathscr{F}_{n-1}^{\Delta Z}$ measurable and that $\E \left[\varepsilon^{*}_{n \theta_{j}}(\theta^*)\right]=\E \left[\varepsilon^{*}_{n \theta_{j'}}(\theta^*)\right]=\E \left[\varepsilon^{*}_{n \theta_{j} \theta_{j'}}(\theta^*)\right]=0.$ The previous formula reads as
\begin{eqnarray*}
&\E\left[\varepsilon^{*}_{n \theta_{j}}(\theta^*)\varepsilon^{*}_{n \theta_{j'}}(\theta^*)\right] \sum_{l,m=1}^k K^{-1}_{l,m} \left(\varphi(u_l-u_m,\eta^*)-\varphi(u_l,\eta^*)\varphi(-u_m,\eta^*)\right)(-u_l^2)+ \\
&2 \E\left[\varepsilon^{*}_{n \theta_{j}}(\theta^*)\varepsilon^{*}_{n \theta_{j'}}(\theta^*)\right] \sum_{l,m=1}^k K^{-1}_{l,m} \varphi(u_l-u_m,\eta^*)(u_l u_m)  + \\
&\E\left[\varepsilon^{*}_{n \theta_{j}}(\theta^*)\varepsilon^{*}_{n \theta_{j'}}(\theta^*)\right] \sum_{l,m=1}^k K^{-1}_{l,m} \left(\varphi(u_l-u_m,\eta^*)-\varphi(u_l,\eta^*)\varphi(-u_m,\eta^*)\right)(-u_m^2)= \\
&\E\left[\varepsilon^{*}_{n \theta_{j}}(\theta^*)\varepsilon^{*}_{n \theta_{j'}}(\theta^*)\right] \times \\ &\sum_{l,m=1}^k K^{-1}_{l,m}
\left((u^2_l+u^2_m)\varphi(u_l,\eta^*)\varphi(-u_m,\eta^*)-(u_l-u_m)^2\varphi(u_l-u_m,\eta^*)\right)
\end{eqnarray*}
To double check the result note that the last formula gives real matrix since conjugation doest not modify the value of the double sum.

If $j \leq p  < j' \leq p+q,$ then $\left(R^*\right)_{j,j'}$ equals to
\begin{eqnarray*}
&\E \left[\frac{\partial^2}{\partial \theta_{j} \partial \eta_{j'}}\sum_{l,m=1}^k K^{-1}_{l,m} \left(e^{i u_l \varepsilon^*_n(\theta)}-\varphi(u_l,\eta)\right)\left(e^{-i u_m \varepsilon^*_n(\theta)}-\varphi(-u_m,\eta)\right)\right|_{\substack{\theta=\theta^* \\ \eta=\eta^*}} \Bigg]=0,
\end{eqnarray*}
because the differentiation with respect to $\eta_{j'}$ yields a non-random constant  of the form $\varphi_{\eta_j}(u,\eta^*)$ and the differentiation with respect to $\theta_j$ yields the term $\E\left[e^{i u \varepsilon^{*}_n(\theta^*)}iu\varepsilon^{*}_{n \theta_{j}}(\theta^*)\right]=0.$

Finally, if $p < j,j'$ then $\left(R^*\right)_{j,j'}$ equals to
\begin{eqnarray*}
&\E \left[\frac{\partial^2}{\partial \eta_{j} \partial \eta_{j'}}\sum_{l,m=1}^k K^{-1}_{l,m} \left(e^{i u_l \varepsilon^*_n(\theta)}-\varphi(u_l,\eta)\right)\left(e^{-i u_m \varepsilon^*_n(\theta)}-\varphi(-u_m,\eta)\right)\right|_{\substack{\theta=\theta^* \\ \eta=\eta^*}} \Bigg]=\\
&\sum_{l,m=1}^k K^{-1}_{l,m} \left(\varphi_{\eta_j}(u_l,\eta^*)\varphi_{\eta_j'}(-u_m,\eta^*)+\varphi_{\eta_j'}(u_l,\eta^*)\varphi_{\eta_j}(-u_m,\eta^*)\right).
\end{eqnarray*}
To sum it up, $$R^*=
\begin{pmatrix}
W_{\theta \theta}(\theta^*) & 0 \\
0 & W_{\eta \eta}(\eta^*)
\end{pmatrix}
$$ is block diagonal matrix, where
\begin{eqnarray*}
&W_{\theta \theta}(\theta^*)=\E\left[\varepsilon^{*}_{n \theta}(\theta^*)\varepsilon^{T*}_{n \theta}(\theta^*)\right] \times \\
&\sum_{l,m=1}^k K^{-1}_{l,m}
\left((u^2_l+u^2_m)\varphi(u_l,\eta^*)\varphi(-u_m,\eta^*)-(u_l-u_m)^2\varphi(u_l-u_m,\eta^*)\right),
\end{eqnarray*}
and $$ (W_{\eta \eta})_{j,j'}(\eta^*)= \sum_{l,m=1}^k K^{-1}_{l,m} \left(\varphi_{\eta_j}(u_l,\eta^*)\varphi_{\eta_{j'}}(-u_m,\eta^*)+\varphi_{\eta_{j'}}(u_l,\eta^*)\varphi_{\eta_j}(-u_m,\eta^*)\right).$$
\qed

\end{proof}

\textit{Remark 1}: Note that the expression for $(W_{\eta
\eta})_{j,j'}(\eta^*)$ is identical to what we would obtained for i.i.d. samples following \cite{CARRASCO-CGMM}.

\textsl{Remark 2}: Since we have $w \ge 0$, the expression for $w$
yields a non-trivial inequality for characteristic functions.

The next step in calculating the asymptotic covariance matrix of ${\hat \theta_N}$ is the computation of
$S^*={\rm Cov}\left(V^*_{N\theta}(\rho^*),V^*_{N\theta}(\rho^*)\right)$. For this we need to introduce the following auxiliary function:
\begin{align*}
\label{eq:DEF_F}
&F(a,b,c,d,\eta)=\\&ab\big[\varphi(a+b+c+d,\eta)-
\varphi(a+b+c,\eta)\varphi(d,\eta)- \\
&\varphi(a+b+d,\eta)\varphi(c,\eta)+
\varphi(a+b,\eta)\varphi(c,\eta)\varphi(d,\eta)\big].
\end{align*}
\begin{theorem}\label{lemma:cov_theta} Under Conditions 1,2,3 we have
\begin{equation*}
S^*= {\rm Cov} \left(V^*_{N\theta}(\rho^*),V^*_{N\theta}(\rho^*)\right)=\begin{pmatrix}
{\rm Cov} \left(V^*_{N\theta}(\rho^*),V^*_{N\theta}(\rho^*)\right) & 0 \\
0 & {\rm Cov} \left(V^*_{N\eta}(\rho^*),V^*_{N\eta}(\rho^*)\right) \end{pmatrix},
\end{equation*}
where
${\rm Cov} \left(V^*_{N\theta}(\rho^*),V^*_{N\theta}(\rho^*)\right)=s~\E\left[\varepsilon^{*}_{n \theta}(\theta^*)\varepsilon^{T*}_{n \theta}(\theta^*)\right],$
with
\begin{align*}\nonumber
&s=\sum_{l,m,s,t=1}^N K^{-1}_{l,m}K^{-1}_{s,t} \times \\
&\big[F(u_l,u_s,-u_m,-u_t,\eta^*)+
F(u_l,-u_t,-u_m,u_s,\eta^*)+ \\
&F(-u_m,u_s,u_l,-u_t,\eta^*)+
F(-u_m,-u_t,u_l,u_s,\eta^*)\big],
\end{align*}
and
\begin{equation*}
\begin{split}
&\left({\rm Cov} \left(V^*_{N\eta}(\rho^*),V^*_{N\eta}(\rho^*)\right)\right)_{j,j'}=\sum_{l,m,s,t=1}^{N} K^{-1}_{l,m} K^{-1}_{s,t} \times \\ &\Big(\varphi_{\eta_{j}}(u_l,\eta^*)\varphi_{\eta_{j'}}(u_s,\eta^*)\varphi(-u_m-u_t,\eta^*)+
\varphi_{\eta_{j}}(u_l,\eta^*)\varphi_{\eta_{j'}}(-u_t,\eta^*)\varphi(-u_m+u_s,\eta^*)+ \\
&\varphi_{\eta_{j}}(-u_m,\eta^*)\varphi_{\eta_{j'}}(u_s,\eta^*)\varphi(u_l-u_t,\eta^*)+
\varphi_{\eta_{j}}(-u_m,\eta^*)\varphi_{\eta_{j'}}(-u_t,\eta^*)\varphi(u_l+u_s,\eta^*)\Big)
\end{split}
\end{equation*}
\end{theorem}

The proof of the last theorem is a simple calculation like the previous one and the proof uses that $\E\left[\varepsilon^*_{n \theta_i}(\theta^*)\varepsilon^*_{m \theta_i}(\theta^*)\right]=\E\left[\E\left[\varepsilon^*_{n \theta_i}(\theta^*)\varepsilon^*_{m \theta_i}(\theta^*)| \mathscr{F}_{n-1}^{\Delta Z}\right]\right]=\E\left[\varepsilon^*_{n \theta_i}(\theta^*)\left[\E\left[\varepsilon^*_{m \theta_i}(\theta^*)| \mathscr{F}_{n-1}^{\Delta Z}\right]\right]\right]=0$ for $m>n.$


The proof follows the line of arguments for Lemma \ref{lemma:exp}. We note that calculations are considerably simplified if we take $K=I$. Note that both $R^*$ and $S^*$ are of the form $c~\Sigma_{P},$ where $\Sigma_{P}$ is the asymptotic covariance matrix for the prediction error method, see below ($\ref{eq:sigmap1}$), and $c$ is a constant. The last two theorems and Theorem \ref{th:difference} together gives an exact formula for the asymptotic covariance matrix of the estimator.
\begin{theorem}
\label{th:AS_COV}
Under Conditions 1,2 and 3 the asymptotic covariance matrix of the ECF estimator for $\theta^*$ can be written as
%
%
\begin{equation}
\Sigma_{E}=(R^*)^{-1}S^*(R^*)^{-1}=\frac{s}{w^2}~\Sigma_{P},
\end{equation}
where the $s$ and $w$ are given in
%
%
Theorems \ref{lemma:cov_theta} and \ref{lemma:R*}.
\end{theorem}


\section{Combining PE and ECF estimators}
\label{sec:mixed}
In this section we estimate the dynamics in a natural way and then we estimate the noise parameters using the ECF method. We identify $\theta^*$
using only the orthogonality of $\Delta Z$ by applying a
prediction error method. This way we get an estimation
$\hat{\theta}_{N}$  of $\theta^*$, without using the characteristic function of $\Delta Z.$ Then we apply an ECF method with the score
function
$$h_n(u,\eta)=e^{iu\varepsilon_n(\hat{\theta}_{N})}-\varphi(u,\eta)$$
to estimate $\eta^*.$

First, we define the estimated innovation process as in the
previous sections. The prediction error method is obtained by minimizing the cost
function
\begin{equation*}
V_{P,N}(\theta)=\frac{1}{2}\sum_{n=1}^N \varepsilon^2_n(\theta).
\end{equation*}
In practice the estimated $\hat{\theta}_N$ is defined as the solution of
\begin{equation*}
V_{P,N \theta}(\theta)=\sum_{n=1}^N \varepsilon_n(\theta) \varepsilon_{n\theta}(\theta)=0.
\end{equation*}
The asymptotic cost function associated with the PE method is defined as

$$W_{P}(\theta)=\frac{1}{2}\lim_{n \rightarrow \infty} \E \varepsilon_n^2(\theta)=\frac{1}{2} \E \varepsilon^{*2}_n(\theta),$$ recall that
$\varepsilon^*_n(\theta)$ is the innovation process that is
calculated with stationary initial values.
We have
$$
W_{P,\theta}(\theta^*)=0 \text{ and } R_P^*:=W_{P,\theta\theta}(\theta^*)=\E \left[\varepsilon^*_{n \theta}(\theta^*)\varepsilon^{T*}_{n \theta}(\theta^*)\right].
$$
The asymptotic covariance matrix of the PE estimate of $\theta^*$
is given by
\begin{equation}\label{eq:sigmap1}
\Sigma_{P}=\left(\E\left[\varepsilon^*_{n\theta}(\theta^*)\varepsilon^{T*}_{n\theta}(\theta^*)\right]\right)^{-1}.
\end{equation}
An ideal score function for the ECF method to estimate $\eta^*$
would be defined by
\begin{align}
h_{opt,n}(u,\eta)=e^{iu\varepsilon^*_n(\theta^*)}-\varphi(u,\eta).
\end{align}
Since we are not given $\theta^*$ we define an alternative,
$\theta$-dependent score function via
\begin{equation*}
h_n(u,\theta,\eta)=e^{iu\varepsilon_n(\theta)}-\varphi(u,\eta).
\end{equation*}
These are appropriate score functions since $\E \left[h^*_n(u,\theta^*,\eta^*)\right] =0.$

Fix a set of real numbers $u_1,\cdots,u_k$, with $Nk \ge  {\rm
dim}~\eta$ and define
$$h_n(\theta,\eta)=\left(h_n(u_1,\theta,\eta),\cdots,h_n(u_k,\theta,\eta)\right)^T.$$
Then we obtain the estimate $\hat{\eta}_N$ of $\eta^*$ by finding a least squares solution to the over-determined system of equations
$$
h_n(\hat{\theta}_N,\eta) =0 \quad n=1,\ldots,N
$$
More precisely, define the $\theta$-dependent cost function
$$
V_{E,N}(\theta,\eta)=\sum_{n=1}^{N}\left|K^{-1}h_n(\theta,\eta)\right|^2,
$$
where $K$ is a symmetric, positive definite weighting matrix. Then we obtain the estimate $\hat{\eta}_N$ of $\eta^*$ by minimizing $V_{E,N}(\hat{\theta}_N,\eta).$

Define the ($\theta$-dependent) asymptotic cost function as


$$W_{E}(\theta,\eta)=\E\left|K^{-1/2}h^*_{n}(\theta,\eta)\right|^2.$$
Let its Hessian w.r.t. $\eta$ at $\eta = \eta^*$ be denoted by
$$R_{E}^{*}=W_{E,\eta\eta}(\theta^*,\eta^*).$$

To formulate our result we need some technical conditions. Conditions 1 and 2 have been already presented in Section \ref{sec:disc_levy}.
Let $\rho$ be the joint parameter i.e. $\rho=(\theta,\eta).$ Let $D_{\rho}$ and $D_{\rho}^*$ be compact domains such that $\rho^* \in D_{\rho}^* \subset \text{int } D_{\rho}$ and $D_{\rho} \subset G_{\rho}.$

\textbf{Condition 3'}
The equations $W_{P,\theta}(\theta)=0,$ and $W_{E,\eta}(\theta^*,\eta)=0$ have a unique solution in $D_{\rho}^*.$

The following lemma, with minor variation, can be found in
\cite{gerencs_arma}.
\begin{lemma}
Under Conditions 1,2,3' we have $\hat{\theta}_N-\theta^*=O^{Q/2}_M(N^{-1/2}).$
\end{lemma}
Our next result characterizes the estimation error of the ECF method for the noise parameter $\eta^*.$
\begin{theorem}\label{thm:mixed}
Under Conditions 1,2 and 3' we have
$$
\hat{\eta}_N-\eta^*=-(R_{E}^{*})^{-1}\frac{1}{N} V_{E,N \eta
}(\eta^*)+O^{Q/(2(p+q))}_M(N^{-1}).
$$
\end{theorem}


The proof is obtained by the very same methods as Theorem \ref{th:difference} combined with the fact that
\begin{equation}\label{eq:dif4}
\left\|W_{\eta\eta}(\theta^*,\eta^*)-W_{\eta\eta}(\hat{\theta}_N,\eta^*)\right\|=O_M^{Q/2}(N^{-1/2}),
\end{equation}
which is implied by $\hat{\theta}_N-\theta^*=O^{Q/2}_M(N^{-1/2}).$ Equation (\ref{eq:dif4}) and equations (\ref{eq:dif1}), (\ref{eq:dif2}) together imply (\ref{eq:dif3}).
%
%

\section{Efficiency of the single term ECF method}
\label{sec:eff}
In view of the efficiency of the ECF method for i.i.d. samples the
question arises what can be achieved by the proposed adaptation of
the ECF method when identifying the dynamics of a linear
stochastic system. We do not have an answer to this general
question, but we will show that the commonly used PE method can be
outperformed by an appropriately calibrated ECF method when the
noise is CGMY. Without loss of generality we may assume that
$${\rm Var} \left(\Delta Z_n\right)=1.$$  Surprisingly, we will see that the ECF method may
outperform the PE method by using a single $u$ sufficiently close
to $0$. Letting $u$ tend to 0 the asymptotic covariance of the ECF
estimate tends to the asymptotic covariance of the PE estimate. On the
other hand, numerical investigations show that increasing the
number of $u$-s used in the ECF method may not improve the
efficiency significantly.
%
%

For $k=1$ the asymptotic covariance of $\hat{\theta}_N$ obtained by
the ECF method is $\lim_{N \rightarrow \infty}N{\rm Cov}(\hat{\theta}_N-\theta^*),$ , which reads as, using Theorems \ref{lemma:cov_theta} and \ref{lemma:R*},
\begin{equation*}
\begin{split}
&\left(\E\left[\varepsilon^*_{n\theta}(\theta^*)\varepsilon^{*T}_{n\theta}(\theta^*)\right]\right)^{-1}\left(-\frac{1}{4u^2}\left(\frac{\varphi(2u)}{\varphi^2(u)}+\frac{\varphi(-2u)}{\varphi^2(-u)}-\frac{2}{\varphi(u)\varphi(-u)}\right)\right).
\end{split}
\end{equation*}
Recall that the asymptotic covariance of $\hat{\theta}_N$ obtained by the PE method is
\begin{equation*}
\Sigma_{P}=\left(\E\left[\varepsilon_{n\theta}(\theta^*)\varepsilon^T_{n\theta}(\theta^*)\right]\right)^{-1}.
\end{equation*}
Thus the ECF estimator outperforms the PE estimator if
\begin{equation*}
\label{eq:var_ineq}
\frac{s}{w^2}=-\frac{1}{4u^2}\left(\frac{\varphi(2u)}{\varphi^2(u)}+\frac{\varphi(-2u)}{\varphi^2(-u)}-\frac{2}{\varphi(u)\varphi(-u)}\right)<1.
\end{equation*}
%
%
\begin{theorem}
For all $u \neq 0$, sufficiently close to $0$ we have $\frac{s}{w^2} <
1$, and thus the corresponding single-term ECF estimator of the
system parameter $\theta^*$, with $k=1$, outperforms the PE
estimator.
\end{theorem}

\begin{proof}
First note that for
\begin{equation}
g(u)=-\left(\frac{\varphi(2u)}{\varphi^2(u)}+\frac{\varphi(-2u)}{\varphi^2(-u)}-\frac{2}{\varphi(u)\varphi(-u)}\right)
\end{equation}
$g(u)=\overline{g(u)}$ holds, so $g$ is a real-valued function. Let us compute the Taylor expansion of $g$ around 0. The first three derivatives of $\varphi(u)$ for a CGMY
process with zero expectation are given by
\begin{align}\nonumber
\varphi(0)&=1, \\ \nonumber
\varphi_{u}(0)&=i \E\left[\Delta Z_n\right]=0, \\ \nonumber
\varphi_{uu}(0)&=-\E\left[(\Delta Z_n)^2\right]= \\&=-C\Gamma(2-Y)\left(M^{Y-2}+G^{Y-2}\right)=-1, \nonumber \\ \nonumber
\varphi_{uuu}(0)&=-i\E\left[(\Delta Z_n)^3\right]=0. \nonumber
\end{align}
After a lengthy computation, that we omit, we get that
\begin{equation}
g(u)=-4u^2+\frac43 G^{-2}(Y-2)(Y-3)u^4+{\cal O} (u^6).
\end{equation}
Thus
\begin{equation*}
\begin{split}
\frac{s}{w^2}=-\frac{1}{4u^2}\left(\frac{\varphi(2u)}{\varphi^2(u)}+\frac{\varphi(-2u)}{\varphi^2(-u)}-\frac{2}{\varphi(u)\varphi(-u)}\right)= \\
1-\frac13 G^{-2}(Y-2)(Y-3)u^2+{\cal O} (u^4).
\end{split}
\end{equation*}
%
%
Since $G<0$ and $0<Y<2,$ the coefficient of $u^2$ is negative. Hence, by choosing $u$ sufficiently small $\frac{s}{w^2}<1$ can be achieved. $\square$
\end{proof}
Numerical investigations show that for a CGMY process with parameters $C=0.564, G=M=1, Y=0.5$ the minimal value of $g$ is approximately $0.73$. We experienced that increasing the number of $u$-s that are used does not reduce $s/w^2$ significantly. For example, choosing $(u_1,\ldots,u_k)=(0.1,0.2,\ldots,0.1k)$ and $K=I$ we get $s/w^2=0.688.$

\section{Discussion}
In the previous section we assumed that $\E\left[\Delta
Z_n\right]=0.$ This is a standard assumption in system
identification, but certainly not realistic for financial data.
Thus e.g. in the case of a CGMY process this assumption would
imply $G=M$, excluding possible skewness in the distribution.
While the case $\E \left[\Delta Z_n\right]=m^* \neq 0$ would pose
no problem for the case of i.i.d. data, surprisingly the single
term ECF method may break down. The reason for this is that
$V^*_{N \theta}(\theta,\eta)$ is no more a score function, since
we cannot guarantee that
$$\E \left[V^*_{N
\theta}(\theta^*,\eta)\right]=0
$$
holds, see Lemma \ref{lemma:exp}. Namely in the proof of Lemma
\ref{lemma:exp} we make use of the equality
\begin{equation}
\mathbb{E} \left[\varepsilon^*_{n \theta}(\theta^*)\right] = 0,
\end{equation}
which may not be valid. Note, however, that $$
h_n(u;\theta,\eta)= e^{iu \varepsilon_n (\theta)}-\varphi(u,\eta)
$$ does have the property required for a score function, namely
\begin{equation}
\mathbb{E} \left[e^{iu \varepsilon^*_n
(\theta^*)}-\varphi(u,\eta^*) \right] = 0.
\end{equation}

Thus, using an instrumental variable approach, we may choose an
appropriate linear combination of these score functions, say $
\sum_{n=1}^N  M h_n(\theta,\eta),$ where $M$ is a $(p+r) \times
k $ matrix, and consider the equation:
$$ \sum_{n=1}^N  M
h_n(\theta,\eta) = 0.
$$
Assuming that $p+r < k$ we may rightly expect that taking
mathematical expectation the resulting equation has $(\theta^*,
\eta^*)$ as an {\it isolated} solution, and we may proceed as in
Section 5. The elaboration of the details is the subject of
ongoing research.

An alternative approach is to adapt our method of combining the PE
method with the ECF method. For this we first need to extend the
PE method to deal with the case $m^* \neq 0,$ which is a standard
exercise. Write $\Delta Z_n=\Delta e_n+m^*,$ where $\E
\left[\Delta e_n\right]=0.$ Then
equation (\ref{eq:disc_levy}) reads as
\begin{equation*}
\Delta y_n=A(\theta^*)\left(\Delta e_n+m^*\right).
\end{equation*}
Define the estimated innovation process by
\begin{equation*}
\varepsilon_n(\theta)=A^{-1}(\theta)A(\theta^*)(\Delta e_n+m^*).
\end{equation*}
Clearly $\E \left[\varepsilon_n(\theta^*)\right]=m^*,$ thus we define the cost function via
\begin{equation*}
V_N(\theta,m)=\frac{1}{2} \sum_{n=1}^{N} \left(\varepsilon_n(\theta)-m\right)^2.
\end{equation*}
The estimate $(\hat{\theta}_N,\hat{m}_N)$ of $(\theta^*,m^*)$ is obtained by solving
\begin{equation*}
\frac{\partial}{\partial (\theta,m)}V_N(\theta,m)=0,
\end{equation*}
which can be written as
\begin{align*}
0&=V_{N \theta}(\theta,m)=\sum_{n=1}^{N} \left(\varepsilon_n(\theta)-m\right)\varepsilon_{n \theta}(\theta) \\
0&=V_{N m}(\theta,m)=-\sum_{n=1}^{N} \left(\varepsilon_n(\theta)-m\right).
\end{align*}
Having estimated the system dynamics with this extended PE method
one may estimate the noise parameters with the ECF method, as in
Section $\ref{sec:mixed}.$

The shortcoming of the above approach is that it does not exploit
fully the potentials of the ECF method in estimating the system
dynamics. Therefore we suggest a second pass for estimating
$\theta^*$ via a single term ECF method, with $\hat{\eta}_N$
considered as the true parameter, applied to the system
\begin{equation}
\Delta y_n-A(\hat{\theta}_N)\hat{m}_N= A(\theta^*)(\Delta
Z_n-\hat{m}_N),
\end{equation}
where $\hat{m}_N,\hat{\theta}_N,\hat{\eta}_N$ are the first
estimates. Define $\Delta \tilde{y}_n=\Delta y_n-A(\hat{\theta}_N)\hat{m}_N$ and $\Delta \tilde{Z}_n$ the previous equation reads as
\begin{equation}
\Delta \tilde{y}_n=A(\theta^*)\Delta \tilde{Z}_n,
\end{equation}
with $\E \left[ \Delta \tilde{Z}_n\right]=0.$
Thus we may proceed according to Section $\ref{sec:ECF}$ to obtain the corrected estimate of $\theta^*.$


What we have obtained is an extension of the single term ECF
method, which is computationally simpler. Ongoing investigations
suggest that the efficiency of this generalized single term ECF
method is as good as the original single term ECF when $m^*=0.$


Finally we mention one more very different approach to deal with
the problem of non-zero expectation, having interest on its own.
The idea is to use an ECF method directly for blocks of
unprocessed data, i.e. for blocks of the time series $(y_n).$ For
this purpose let us imbed our data into the class of time series
$$\Delta y_n(\theta,\eta)=A(\theta)\Delta Z_n(\eta).$$
Note that for $(\theta,\eta) = (\theta^*,\eta^*)$ we recover (in a statistical sense) our observed data.
Fix a block length, say $r,$ and define the $r$-dimensional blocks
$$\Delta Y^r_n(\theta,\eta)=(\Delta y_n(\theta,\eta),\ldots,\Delta y_{n+r-1}(\theta,\eta)).$$
Letting $U$ be an arbitrary $r$-vector the
characteristic function of $\Delta Y^r_n (\theta,\eta)$ is given
by
$$\varphi_n(U,\theta,\eta)=\E \left[e^{iU^T \Delta Y^r_n(\theta,\eta)}\right],$$
and the corresponding score function will be defined as
$$
h_n(U,\theta,\eta)=e^{iU^T\Delta Y^r_n}-\varphi_n(U,\theta,\eta).
$$
The point is that the characteristic function can be explicitly
computed, at least in theory, as
\begin{equation}
\begin{split}
\varphi_n(U,\theta,\eta)=\E \left[\exp\{iU^T \Delta Y^r_n(\theta,\eta)\}\right]= \\
\E \left[\exp \left\{ i\sum_{j=1}^r U_j \sum_{l=0}^{\infty} h_l(\theta) \Delta Z_{n+j-1-l}(\eta)\right\}\right]= \\
\prod_{j=0}^{\infty} \varphi_{\Delta Z(\eta)}(v_j(\theta)),
\end{split}
\end{equation}
with some $\theta$-dependent constants $v_j$. Here
$\varphi_{\Delta Z}$ denotes the characteristic function of
$\Delta Z_1(\eta).$
%
The weakness of this approach is that the characteristic function
$\varphi_n(U,\theta,\eta)$ is given in terms of an infinite
product, therefore it is not clear how to use it in actual
computations.

\section{Appendix}

Let $\theta$ be a $d$-dimensional parameter vector.
\begin{definition}
We say that $x_n(\theta)$ is $M$-bounded of order Q if for all $1\leq q \leq Q$,
$$ M^Q_q(x)=\sup_{n>0, \theta \in D} \E^{1/q}\left|x_n(\theta)\right|^q < \infty $$
\end{definition}

Define $\mathscr{F}_n=\sigma\left\{e_i: i \leq n \right\}$ and $\mathscr{F}^+_n=\sigma\left\{e_i: i > n \right\}$ where $e_i$-s are i.i.d. random variables.

\begin{definition}
We say that a stochastic process $\left(x_n(\theta)\right)$ is $L$-mixing of order $Q$ with respect to $\left(\mathscr{F}_n,\mathscr{F}^+_n\right)$ uniformly in $\theta$ if it is $\mathscr{F}_n$ progressively measurable, M-bounded of order $Q$ with any positive $r$ and
$$\gamma_q(r,x)=\gamma_q(r)=\sup_{n \geq r, \theta \in D} \E^{1/q} \left|x_n(\theta)-\E\left[x_n(\theta)|\mathscr{F}^+_{n-r}\right]\right|^q,$$
we have for any $1 \leq q \leq Q,$
$$\Gamma_q(x)=\sum_{r=1}^{\infty} \gamma_q(r) < \infty.$$
\end{definition}

\begin{theorem}\label{thm:ineq}
Let $(u_n), n \geq 0$ be an $L$-mixing process of order $Q$ with $\E u_n=0$ for all $n,$ and let $(f_n)$ be a deterministic sequence. Then we have for all $1 \leq m \leq Q/2,$
\begin{align}
\E^{1/(2m)}\left|\sum_{n=1}^N f_n u_n\right|^{2m} \leq C_m \left(\sum_{n=1}^N f^2_n\right)^{1/2} M^{1/2}_{2m}(u) \Gamma^{1/2}_{2m}(u)
\end{align}
where $C_m=2(2m-1)^{1/2}.$
\end{theorem}

Define $$\Delta x/\Delta^{\alpha} \theta=\left|x_n(\theta+h)-x_n(\theta)\right|/\left|h\right|^{\alpha}$$
for $n\geq0, \theta\neq\theta+h \in D$ with $0<\alpha \leq 1.$

\begin{definition}
We say that $x_n(\theta)$ is $M$-H\"{o}lder continuous of order $Q$ in $\theta$ with exponent $\alpha$ if the process $\Delta x/\Delta^{\alpha} \theta$ is $M$-bounded of order $Q.$
\end{definition}

Now let us suppose that $(x_n(\theta))$ is measurable, separable, $M$-bounded of order $Q$ and $M$-H\"{o}lder of order $Q$ in $\theta$ with exponent $\alpha$ for $\theta \in D.$ The realizations of $(x_n(\theta))$ are continuous in $\theta$ almost surely hence $$x^*_n=\max_{\theta \in D_0} \left|x_n(\theta)\right|$$ is well defined for almost all $\omega,$ where $D_0 \subset \text{int } D$ is a compact domain. Since the realizations of $(x_n(\theta))$ are continuous, $x^*_n$ is measurable with respect to $\mathscr{F}.$

\begin{theorem}\label{thm:moments}
Assume that $(x_n(\theta))$ is measurable, separable, $M$-bounded of order $Q$ and $M$-H\"{o}lder of order $Q$ in $\theta$ with exponent $\alpha$ for $\theta \in D.$ Then we have for all positive $q \leq Q\alpha/s$ and $p/\alpha <s \leq Q/q,$
$$M_q(x^*)\leq C \left(M_{qs}(x)+M_{qs}(\Delta x/\Delta^{\alpha} \theta)\right)$$
where $C$ depends only on $p, q, s, \alpha$ and $D_0, D.$
\end{theorem}
Choosing $f_n=1$ and $\alpha=1$ and using Theorem \ref{thm:ineq} and \ref{thm:moments} we obtain

\begin{theorem}\label{thm:sup}
Let $(u_n(\theta))$ be an $L$-mixing of order $Q$ uniformly in $\theta \in D$ such that $\E u_n(\theta)=0$ for all $n \geq 0, \theta \in D,$ and assume that $\Delta u /\Delta \theta$ is also $L$-mixing of order $Q,$ uniformly in $\theta, \theta+h \in D.$ Then
\begin{equation}
\sup_{\theta \in D_0} \left|\frac{1}{N} \sum_{n=1}^N u_n(\theta)\right|=O_M^{Q/p}(N^{-1/2})
\end{equation}
\end{theorem}

\begin{theorem}\label{thm:implicit}
Let $D_0$ and $D$ be as above. Let $W_{\theta}(\theta), \delta W_{\theta}(\theta), \ \theta \in D \subset \mathbb{R}^p$ be $\mathbb{R}^p$-valued continuously differentiable functions, let for some $\theta^* \in D_0, W_{\theta}(\theta^*)=0,$ and let $W_{\theta \theta}(\theta^*)$ be nonsingular. Then for any $d>0$ there exists positive numbers $d',d''$ such that
\begin{equation}
\left|\delta W_{\theta}(\theta)\right|<d' \text{ and } \left\|\delta W_{\theta \theta}(\theta)\right\|<d''
\end{equation}
for all $\theta \in D_0$ implies that the equation $W_{\theta}(\theta)+\delta W_{\theta}(\theta)=0$ has exactly one solution in a neighborhood of radius $d$ of $\theta^*.$
\end{theorem}

\end{document}